\documentclass[journal,twoside,web]{ieeecolor}

\usepackage{cite}
\usepackage{amsmath,amssymb,amsfonts}
\usepackage{hyperref}
\usepackage{epstopdf}
\usepackage{textcomp}
\usepackage{graphicx,color}
\usepackage{mathrsfs}
\usepackage[vlined,ruled]{algorithm2e}
\usepackage{subfigure}
\usepackage{url}
\usepackage{color}
\usepackage{dsfont}
\usepackage{bbm}
\usepackage{booktabs}
\usepackage{array}
\usepackage[table]{xcolor}
\usepackage{yfonts}
\usepackage{cleveref} 
\usepackage{tikz}
\usepackage{pgfplots}
\usepackage{multirow}
\usepackage{textcomp}
\usepackage{tabularx}
\usepackage{placeins}
\usepackage{float}
\usepackage{nccmath}
\usepackage{accents}
\usepackage{multicol}
\usepackage{etoolbox, refcount}
\usepackage{chngcntr}
\usepackage{apptools}
\usepackage{relsize}
\usepackage[font=footnotesize]{caption}
\usepackage{textcomp}

\newtheorem{theorem}{Theorem}[section]
\newtheorem{lemma}[theorem]{Lemma}
\newtheorem{definition}{Definition}
\newtheorem{corollary}[theorem]{Corollary}
\newtheorem{proposition}[theorem]{Proposition}
\newtheorem{problem}{Problem}
\newtheorem{remark}{Remark}

\definecolor{subsectioncolor}{RGB}{1,1,0}

\newcommand{\map}[3]{#1: #2 \rightarrow #3}
\newcommand{\setdef}[2]{\{#1 \; : \; #2\}}

\newcommand{\real}{\mathbb{R}}

\newcommand{\integerspos}{\mathbb{Z}_{> 0}}

\newcommand{\Cc}{\mathcal{C}}
\newcommand{\Dc}{\mathcal{D}}

\newcommand{\argmin}[2] {\mathrm{arg}\min_{#1}#2}

\DeclareSymbolFont{bbold}{U}{bbold}{m}{n}
\DeclareSymbolFontAlphabet{\mathbbold}{bbold}

\newcommand{\norm}[1]{\left\lVert#1\right\rVert}

\newcommand{\nom}{\operatorname{nom}}
\newcommand{\safe}{\operatorname{safe}}
\newcommand{\stable}{\operatorname{stable}}

\newcommand\oprocendsymbol{\hbox{$\bullet$}}
\newcommand\oprocend{\relax\ifmmode\else\unskip\hfill\fi\oprocendsymbol}

\renewcommand{\theenumi}{(\roman{enumi})}

\newcommand*{\QEDA}{\hfill\ensuremath{\blacksquare}}%

\newcommand\xqed[1]{%
  \leavevmode\unskip\penalty9999 \hbox{}\nobreak\hfill
  \quad\hbox{#1}}
\newcommand\demo{\xqed{$\bullet$}}

\newcommand\problemfinal{\xqed{$\triangle$}}

\newcounter{countitems}
\newcounter{nextitemizecount}
\newcommand{\setupcountitems}{%
  \stepcounter{nextitemizecount}%
  \setcounter{countitems}{0}%
  \preto\item{\stepcounter{countitems}}%
}
\makeatletter
\newcommand{\computecountitems}{%
  \edef\@currentlabel{\number\c@countitems}%
  \label{countitems@\number\numexpr\value{nextitemizecount}-1\relax}%
}
\newcommand{\nextitemizecount}{%
  \getrefnumber{countitems@\number\c@nextitemizecount}%
}
\newcommand{\previtemizecount}{%
  \getrefnumber{countitems@\number\numexpr\value{nextitemizecount}-1\relax}%
}
\makeatother    
{\end{itemize}%
\unskip\computecountitems\ifnumcomp{\previtemizecount}{>}{4}{\end{multicols}}{}}

\newcommand{\longthmtitle}[1]{\mbox{}\emph{(#1):}}

\renewcommand{\baselinestretch}{1}
\newcommand{\comment}[1]{} 
\newcolumntype{P}[1]{>{\centering\arraybackslash}p{#1}}
\allowdisplaybreaks

\title{\bf Optimization-Based Safe Stabilizing Feedback
  \\
  with Guaranteed Region of Attraction}

\author{Pol Mestres \qquad Jorge Cort{\'e}s \thanks{P. Mestres and
    J. Cort\'es are with the Department of Mechanical and Aerospace
    Engineering, University of California, San Diego,
    \{pomestre,cortes\}@ucsd.edu. This work was supported by NSF Award
    IIS-2007141.}}%

\begin{document}

\maketitle
\thispagestyle{empty}
\pagestyle{empty}

\begin{abstract}
  This paper proposes an optimization with penalty-based feedback
  design framework for safe stabilization of control affine
  systems. Our starting point is the availability of a control
  Lyapunov function (CLF) and a control barrier function (CBF)
  defining affine-in-the-input inequalities that certify,
  respectively, the stability and safety objectives for the dynamics.
  Leveraging ideas from penalty methods for constrained optimization,
  the proposed design framework imposes one of the inequalities as a
  hard constraint and the other one as a soft constraint. We study the
  properties of the closed-loop system under the resulting feedback
  controller and identify conditions on the penalty parameter to
  eliminate undesired equilibria that might arise. Going beyond the
  local stability guarantees available in the literature, we are able
  to provide an inner approximation of the region of attraction of the
  equilibrium, and identify conditions under which the whole safe set
  belongs to it. Simulations illustrate our results.
\end{abstract}


\section{Introduction}\label{sec:introduction}

\IEEEPARstart{S}afety-critical control has garnered a lot of attention
in the controls and robotics communities motivated by applications to
many different classes of engineered and natural systems. Safety
refers to the ability to ensure by design that the evolution of the
dynamics stays within a desired set.  Control barrier functions (CBFs)
are a useful tool to deal with safety specifications that do not
require addressing the difficult task of computing the system's
reachable set.  In many scenarios, safety must be achieved together
with some stabilization goal, and this raises interesting challenges
for control design in order to ensure that both are achieved via
feedback controllers that are easily implemented and have appropriate
smoothness guarantees. These challenges motivate us to develop here an
optimization with penalty-based feedback design framework for safe
stabilization of control affine systems.

\subsubsection*{Literature Review}
We rely on ideas from two different bodies of work.  The first one is
CLFs~\cite{EDS:98}, which have been successfully used in the
control design for stabilization of nonlinear systems. Of particular
interest to this work is the pointwise-minimum norm (PMN)
formula~\cite{RAF-PVK:96a}, that uses a CLF to compute a stabilizing
controller. The second relevant body of work pertains to
CBFs~\cite{PW-FA:07,ADA-SC-ME-GN-KS-PT:19}, whose aim
is to render a certain predefined safe set forward invariant.
However, in applications where both safety and stability must be
certified, CBFs fall short of providing provable stability
guarantees. To tackle this issue,~\cite{MZR-BJ:16} combines a CLF and
a CBF into a so-called CLBF, and then uses Sontag's universal formula
to derive a smooth controller. However, in general it might be
difficult to satisfy the conditions required for the existence of such
a CLBF~\cite{PB-CMK:17}. Another approach is the universal formula for
smooth safe stabilization from~\cite{PO-JC:19-cdc}.  However, this
formula is only applicable in a set where both the CLF and the CBF are
compatible (i.e., there exists a control satisfying their associated
inequalities at every point of the set).  An alternative
approach~\cite{ADA-SC-ME-GN-KS-PT:19} to tackle joint
safety and stability specifications is to combine the CLF and the CBF
in a quadratic program (QP). To guarantee the feasibility of the
program when the functions are not compatible and to avoid the
resulting controller to be non-Lipschitz when they
are~\cite{BJM-MJP-ADA:15}, the stability constraint is often
relaxed. This results in a lack of guarantee of stability, even for
arbitrarily large penalties in the relaxation parameter~\cite{MJ:18}.
Moreover, as shown in~\cite{MFR-APA-PT:21,XT-DVD:21}, this QP
formulation can introduce undesired equilibria beyond the original
equilibrium, which can even be asymptotically stable. This line of
work~\cite{MJ:18,XT-DVD:21} then identifies conditions under which
local stability guarantees of the equilibrium can be given.
Although the region of attraction is not explicitly
  characterized, a strategy similar to the one pursed here could be
  employed. An alternative design, e.g.,~\cite{LW-ADA-ME:17},
assumes a priori knowledge of a CBF and a nominal (possibly unsafe)
stabilizing controller. Then, a safety filter is applied to this
nominal controller. As a result, the filtered controller generally
lacks stability guarantees. The recent paper~\cite{WSC-DVD:21} gives
an estimate of the region of attraction of the closed-loop system
obtained by using such a filtered controller.

\subsubsection*{Statement of Contributions}
We consider the problem of safe stabilization of control affine
systems.  Given a control Lyapunov function and a control barrier
function whose $0$-superlevel set defines an arbitrary, possibly
non-convex safe set, we aim to synthesize a safe, stabilizing feedback
and identify the region of attraction of the origin for the resulting
closed-loop system. In particular, we study under what conditions such
region of attraction contains the safe set. The contributions of this
paper are the following.  Given the safety and stability objectives,
our first contribution designs an optimization with penalty-based
controller that has one of the objectives as a hard constraint and the
other as a soft constraint.  The controller depends on a penalty
parameter that can be tuned to enhance the soft objective at the cost
of reduced optimality, while guaranteeing the satisfaction of the hard
constraint. An advantage of the proposed design is that the controller
is automatically Lipschitz and has a closed-form expression.  Our
second contribution shows that the controller can introduce undesired
equilibrium points different from the origin.  By choosing the penalty
parameter appropriately, and under some technical conditions, these
undesired equilibria can be eliminated. Finally, our third
contribution shows that the proposed controller can be tuned to
provide an inner approximation of the region of attraction of the
origin for the closed-loop system. As a consequence of this analysis,
we provide conditions under which all of the safe set belongs to the
region of attraction of the origin for the closed-loop
system. Simulations on a planar system compare our design with other
approaches in the literature.

\section{Preliminaries on CLFs and CBFs}

This section presents\footnote{We denote by $\integerspos, \real$ and
  $\real_{\geq0}$ the set of positive integers, real, and nonnegative
  real numbers, resp. We write
  $\text{int}(\mathcal{S}), \partial\mathcal{S}$ for the interior and
  the boundary of the set $\mathcal{S}$, resp.  Given $x\in\real^{n}$,
  $\norm{x}$ denotes its Euclidean norm.  Given
  $f:\real^{n}\to\real^{n}$, $g:\real^{n}\to\real^{n\times m}$ and a
  smooth function $W:\real^{n}\to\real$, the notation
  $\map{L_{f}W}{\real^n}{\real}$
  (resp. $\map{L_gW}{\real^n}{\real^{m}}$) denotes the Lie derivative
  of $W$ with respect to $f$ (resp. $g$), that is
  $L_{f}W=\nabla W^{T}f$ (resp.  $\nabla W^{T} g$). We denote by
  $\Cc^{1}(\real^n)$ and $\Cc^{2}(\real^n)$ the set of continuously
  differentiable and twice continuously differentiable functions in
  $\real^{n}$, respectively.  Given $a \in \real^n$ and $b\in \real$,
  let $H$ denote the hyperplane defined by
  $H = \setdef{x\in\real^n}{\langle a, x\rangle =b}$. We denote the
  projection of $v\in\real^{n}$ onto $H$ by
  $P_{H}(v) =v-\frac{\langle a, v\rangle-b}{\norm{a}^2}a$.  A function
  $\beta:\real\to\real$ is of class $\mathcal{K}$ if $\beta(0)=0$ and
  $\beta$ is strictly increasing. If moreover
  $\lim\limits_{t\to\infty}\beta(t)=\infty$, $\beta$ is of class
  $\mathcal{K}_{\infty}$.  A function $V:\real^{n}\to\real$ is
  positive definite if $V(0)=0$ and $V(x)>0$ for $x \neq 0$. Given a
  matrix $M\in\real^{n\times m}$,
  $\ker(M)=\setdef{x\in\real^m}{Mx=\textbf{0}_n}$. Given a square
  matrix $A\in\real^{n \times n}$ with eigenvectors $\{v_j\}_{j=1}^n$
  and corresponding eigenvalues $\{\lambda_j\}_{j=1}^n$, the stable
  subspace of $A$ is defined as
  $\mathcal{V}_{s}(A)=\text{span}(\setdef{v_j}{\Re{(\lambda_j)}<0, \
    j=1,\hdots,n})$, where $\Re{(\lambda_j)}$ denotes the real part of
  $\lambda_{j}$. We denote by $\bar{\lambda}_{\min}(A)$ and
  $\lambda_{\max}(A)$ the smallest non-zero and largest real parts of
  the eigenvalues of~$A$, respectively.}  preliminaries on control
Lyapunov and barrier functions.  Consider a control-affine system
\begin{align}\label{eq:control-affine-sys}
  \dot{x}=f(x)+g(x)u,
\end{align}
where $f:\real^{n}\to\real^{n}$ and $g:\real^{n}\to\real^{n\times m}$
are locally Lipschitz functions, with $x\in\real^{n}$ the state and
$u\in\real^{m}$ the input. Throughout the paper, and without loss of
generality, we assume $f(0)=0$, so that the origin $x=0$ is the
desired equilibrium point of the (unforced) system.

We start by recalling the notion of Control Lyapunov function
(CLF)~\cite{EDS:98,RAF-PVK:96a}.

\begin{definition}\longthmtitle{Control Lyapunov Function}\label{def:clf}
  Given an open set $\mathcal{D}\subseteq\real^{n}$, with
  $0\in\mathcal{D}$, a continuously differentiable function
  $V:\real^{n}\to\real$ is a \textbf{CLF} on $\mathcal{D}$ for 
  system~\eqref{eq:control-affine-sys} if
  \begin{itemize}
  \item $V$ is proper in $\mathcal{D}$, i.e.,
    $\setdef{x\in\mathcal{D}}{V(x)\leq c}$ is a compact set for all
    $c>0$,
  \item $V$ is positive definite,
  \item there exists a continuous positive function
    $W:\real^{n}\to\real$ such that, for each
    $x\in\mathcal{D}\backslash \{0\}$, there exists a control
    $u\in\real^{m}$ satisfying
    \begin{align}\label{eq:clf-ineq}
      L_fV(x)+L_gV(x)u \leq -W(x).
  \end{align}
\end{itemize}
\end{definition}
\medskip

CLFs provide a way to guarantee asymptotic stability of the
origin. Namely, if a Lipschitz controller $u$
satisfies~\eqref{eq:clf-ineq} for all
$x\in\mathcal{D} \setminus \{0\}$, then the origin of the closed-loop
system is asymptotically stable~\cite{EDS:98}.  If $W(x)$
in~\eqref{eq:clf-ineq} is replaced by $\gamma(V(x))$, where $\gamma$
is a class $\mathcal{K}$ function, then such Lipschitz controller
makes the origin exponentially stable. Such controllers can be
synthesized by means of the pointwise minimum-norm (PMN) control
optimization~\cite[Chapter 4.2]{RAF-PVK:96a},
\begin{align*}
  u(x)
  & = \argmin{u \in\real^{m}}{\frac{1}{2}\norm{u}^2}
  \\
  \notag
    & \qquad \text{s.t.~\eqref{eq:clf-ineq} holds}.
\end{align*}
Note that, at each $x\in \real^n$, this is a quadratic program in~$u$.

Next we recall the notion of Control Barrier Function
(CBF)~\cite{ADA-SC-ME-GN-KS-PT:19}. Let $\mathcal{C}\subseteq\real^{n}$ be
a closed set describing the safe states for the
system~\eqref{eq:control-affine-sys}.

\begin{definition}\longthmtitle{Control Barrier
    Function}
  Let $h:\real^{n}\to\real$ be a continuously differentiable function
  such that $ \mathcal{C}=\setdef{x\in\real^n}{h(x)\geq0}$.  The
  function $h$ is a \textbf{CBF} of $\Cc$ for 
  system~\eqref{eq:control-affine-sys} if there exists a class
  $\mathcal{K}_{\infty}$ function $\alpha$ such that, for all
  $x\in \Cc$, there exists a control $u\in\real^{m}$ satisfying
  \begin{align}\label{eq:cbf-ineq}
    L_fh(x)+L_gh(x)u+\alpha(h(x))\geq0.
  \end{align}
\end{definition}
\smallskip

CBFs can be used to guarantee safety, i.e., forward invariance
of~$\mathcal{C}$ under the
dynamics~\eqref{eq:control-affine-sys}. Namely, if a Lipschitz
continuous controller satisfies~\eqref{eq:cbf-ineq} for all
$x\in\mathcal{C}$, then $\mathcal{C}$ is forward
invariant~\cite[Theorem 2]{ADA-SC-ME-GN-KS-PT:19}. Similar to the PMN
controller above, a common design
methodology~\cite{ADA-SC-ME-GN-KS-PT:19} is via the optimization
\begin{align*}
  u(x)
  & = \argmin{u \in\real^{m}}{\frac{1}{2}\norm{u}^2}
  \\
  \notag
    & \qquad \text{s.t.~\eqref{eq:cbf-ineq} holds},
\end{align*}
which results in a Lipschitz controller~\cite[Theorem
2]{XX-PT-JWG-ADA:15}.

When dealing with both the stability and safety of system trajectories
under the dynamics~\eqref{eq:control-affine-sys}, it is important to
note that an input $u$ might satisfy~\eqref{eq:clf-ineq} but
not~\eqref{eq:cbf-ineq}, or vice versa. The following notion, adapted
from~\cite[Definition 2.3]{PO-JC:19-cdc}, captures when the CLF and
the CBF are compatible.

\begin{definition}\longthmtitle{Compatibility of CLF-CBF pair}
  Let $\mathcal{D} \subseteq \real^{n}$ be open, $\Cc \subset \Dc$
  closed, $V$ a CLF on $\mathcal{D}$ and $h$ a CBF of~$\Cc$. Then, $V$
  and $h$ are \textbf{compatible at $x \in \Cc$} if there exists
  $u\in\real^{m}$ satisfying~\eqref{eq:clf-ineq}
  and~\eqref{eq:cbf-ineq} simultaneously. We refer to both functions
  as compatible if $V$ and $h$ are compatible at every point of $\Cc$.
\end{definition}

\section{Problem Statement}\label{sec:problem-statement}

We are interested in designing controllers that are both stabilizing
and safe. We also require them to be Lipschitz in order to guarantee
existence and uniqueness of solutions of the closed-loop system.
Formally, consider a control-affine system of the
form~\eqref{eq:control-affine-sys}.  Let $\map{V}{\real^n}{\real}$ be
a CLF on the open set $\mathcal{D} \subseteq \real^{n}$ and
$\map{h}{\real^n}{\real}$ be a CBF of the closed set
$\Cc \subset \mathcal{D}$.  We assume the origin belongs to
$\mathcal{C}$.  Given the availability of these functions, it seems
reasonable to employ $V$ to ensure the stabilizing aspect of the
controller and $h$ to ensure safety. We also seek to provide formal
characterizations of the region of attraction of the equilibrium for
the resulting closed-loop system.  If $V$ and $h$ are compatible at
every point in the safe set, one option is to find the control through
pointwise optimization with~\eqref{eq:clf-ineq}
and~\eqref{eq:cbf-ineq} as constraints. However,~\cite{BJM-MJP-ADA:15}
gives a counterexample that shows that this pointwise minimization can
result in a non-Lipschitz controller. To remedy this, and also to
extend the design to scenarios where $V$ and $h$ might not be
compatible at some points in the safe set, a popular
approach~\cite{ADA-SC-ME-GN-KS-PT:19} is to relax one of the
inequalities~\eqref{eq:clf-ineq},~\eqref{eq:cbf-ineq} (in
safety-critical applications, the CLF constraint~\eqref{eq:clf-ineq}),
and formulate a QP that penalizes the relaxation parameter:
\begin{align}\label{eq:clf-cbf-qp}
  u(x)
  & = \argmin{(u,\delta)\in\real^{m+1}}{\frac{1}{2}\norm{u}^2+p\delta^2},
  \\
  \notag
    & \qquad \text{s.t.} \quad \eqref{eq:cbf-ineq}, \ L_fV(x)+L_gV(x) u \leq -W(x)+\delta.
\end{align}
Nevertheless, even in the case where the CLF and the CBF are
compatible at all points in the safe set, the resulting controller
might not be stabilizing even for arbitrarily large values of
$p$~\cite{MJ:18}.  Moreover, as pointed out
in~\cite{MFR-APA-PT:21,XT-DVD:21}, this design might introduce
undesired equilibria in the closed-loop system, which can even be
asymptotically stable. To the best of the authors' knowledge, only
local stability guarantees exist~\cite[Theorem
3]{XT-DVD:21},\cite[Theorem 1]{MJ:18}, and no estimates of the region
of attraction are available in the literature.

An alternative design, e.g.,~\cite{LW-ADA-ME:17},
assumes a nominal (possibly unsafe) stabilizing controller $u_{\nom}$
is available, and seeks to modify it as little as possible while
guaranteeing safety. This can be done by solving the following QP:
\begin{align}\label{eq:unom-qp}
  u(x)
  & =  \argmin{u\in\real^{m}}{\frac{1}{2}\norm{u-u_{\nom}(x)}^2},
  \\
  \notag
  & \qquad \text{s.t.} \quad \eqref{eq:cbf-ineq}.
\end{align}
In general, the resulting modified controller might not retain the
stability properties of the original nominal controller but, under
certain conditions~\cite{WSC-DVD:21}, one can provide an estimate of
the region of attraction of the equilibrium. Interestingly, nominal
controllers other than the given one might result in larger regions of
attraction, so in this sense the design directly with the CLF offers
greater flexibility.

We are interested in building an alternative to the
designs~\eqref{eq:clf-cbf-qp},~\eqref{eq:unom-qp} to solve the
aforementioned issues.  In particular, we tackle the following problem:

\begin{problem}\label{pb:problem-1}
  Determine a Lipschitz control law $u$ and a region of
  attraction $\Gamma\subseteq\real^{n}$, $\Gamma\cap\mathcal{C}\neq0$
  such that for all $x(0)\in\Gamma\cap\mathcal{C}$,
  $x(t)\in\mathcal{C}$ for all $t\geq0$ and the
  system~\eqref{eq:control-affine-sys} in closed-loop with $u$ is
  asymptotically stable with respect to the origin.  \problemfinal
\end{problem}

\section{Safety and Stability via QP with
  Penalty}\label{sec:design}

In this section we design a candidate control law to solve
Problem~\ref{pb:problem-1} by leveraging the
CLF $V$ and the CBF~$h$. We first present our exposition in a general
context, then particularize to our setting.  Consider general
Lipschitz functions $a,c : \real^{n}\to\real$ and
$b,d : \real^{n}\to\real^{m}$. Consider the following two affine
inequalities in $u\in\real^{m}$,
\begin{align*}
  a(x)+b(x)u\leq0 , \quad
  c(x)+d(x)u\leq0 .  
\end{align*}
Given a neighborhood $\bar{\Cc}$ of $\Cc$, we assume
  that for every $x\in\bar{\Cc}$, there exist $u_1,u_{2}\in\real^{m}$
  such that $a(x)+b(x)u_1\leq0$ and $c(x)+d(x)u_2\leq0$.
To select $u$, we regard at the first inequality as a \textit{soft
  constraint} and the second as a \textit{hard constraint}.
Inspired by the theory of penalty methods for constrained
optimization~\cite[Chapter 13]{DGL:84}, we formulate a QP where we
include the soft constraint in the objective function with a penalty
parameter ($\epsilon>0$) and enforce the hard constraint.  The
resulting solution of the QP is parametrized by $x\in\real^{n}$ and
$\epsilon$:
\begin{align}\label{eq:usafe-qp}
  \notag
  u_{\epsilon}(x)
  & :=  \argmin{u\in\real^m}
  {\frac{1}{2}\norm{u}^2}+\frac{1}{\epsilon}(a(x)+b(x)u),
  \\
  & \qquad \text{s.t.} \quad c(x)+d(x)u\leq0.
\end{align}
Since this optimization problem is a QP, it is convex.  The following
result gives a closed-form expression for $u_{\epsilon}$ and
establishes that it is Lipschitz.

\begin{proposition}\longthmtitle{Closed-form expression for Lipschitz
    controller}\label{prop:usafe-formula}
  Let $a,c : \real^{n}\to\real$ and $b,d : \real^{n}\to\real^{m}$ be
  Lipschitz, $\bar{\Cc}$ a neighborhood of $\Cc$ and
    assume that for every $x\in\bar{\Cc}$, there exist
    $u_1, u_{2}\in\real^{m}$ such that $a(x)+b(x)u_1\leq0$ and
    $c(x)+d(x)u_2\leq0$. For each $x \in \Cc$, let
  $H(x):=\setdef{u\in\real^m}{c(x)+d(x)u=0}$ and
  $e(x) := c(x)-\frac{1}{\epsilon}d(x)b(x)$. Then,
  \begin{align}\label{eq:usafe}
    u_{\epsilon}(x)=
    \begin{cases}
      -\frac{1}{\epsilon}b(x) & \text{if } e(x)\leq0,
      \\
      P_{H(x)}(-\frac{1}{\epsilon}b(x)) & \text{if } e(x) >0,
    \end{cases}
  \end{align}
  and $u_{\epsilon}$ is Lipschitz on
  $\bar{\Cc}\backslash\{0\}$. Moreover, if $d(0)\neq0$, $u_{\epsilon}$
  is Lipschitz at $0$.
\end{proposition}
\begin{proof}
  The expression~\eqref{eq:usafe} follows by calculating the KKT
  points of~\eqref{eq:usafe-qp}.  Note that~\eqref{eq:usafe} is well
  defined because if $d(x)=0$, necessarily
  $e(x)=c(x)\leq0$. Lipschitzness of $u_{\epsilon}(x)$ follows
  from~\cite[Section 3.10, Theorem 2]{DGL:69}, which as a special case
  includes the minimization of a quadratic cost function subject to
  affine inequality constraints.
\end{proof}

We next particularize the general design~\eqref{eq:usafe-qp} to our
setup. We consider two cases:
\\
\textbf{Safety QP with stability penalty:} The selection
$a(x)=L_{f}V(x)+W(x)$, $b(x)=L_{g}V(x)$,
$c(x)=-L_{f}h(x)-\alpha(h(x))$, and $d(x)=-L_{g}h(x)$ makes the CLF
inequality~\eqref{eq:clf-ineq} a soft constraint and the CBF
inequality~\eqref{eq:cbf-ineq} a hard one. We denote by
$u_{\epsilon}^{\safe}$ the controller resulting
from~\eqref{eq:usafe-qp}. If $L_gh(0)\neq0$,
Proposition~\ref{prop:usafe-formula} guarantees that
$u_{\epsilon}^{\safe}$ is Lipschitz on $\Cc$. Moreover, since it
satisfies the CBF inequality~\eqref{eq:cbf-ineq} for all $x \in \Cc$,
the resulting closed-loop system is safe for all $\epsilon>0$;
\\
\textbf{Stability QP with safety penalty:} Alternatively, the
selection $a(x)=-L_{f}h(x)-\alpha(h(x))$, $b(x)=-L_{g}h(x)$,
$c(x)=L_{f}V(x)+W(x)$, and $d(x)=L_{g}V(x)$, makes the CBF
inequality~\eqref{eq:cbf-ineq} a soft constraint and the CLF
inequality~\eqref{eq:clf-ineq} a hard one.  We denote by
$u_{\epsilon}^{\stable}$ the resulting controller
from~\eqref{eq:usafe-qp}. In this case, $d(0)=0$ and hence
Proposition~\ref{prop:usafe-formula} only guarantees that
$u_{\epsilon}^{\stable}$ is Lipschitz in
$\bar{\Cc}\backslash\{0\}$. Moreover, since~\eqref{eq:clf-ineq} is
satisfied for all $x\in\bar{\Cc}\backslash\{0\}$, the origin is
asymptotically stable for the resulting closed-loop system.

From this point onwards, we formulate the results for the controller
$u_{\epsilon}^{\safe}$. With minor modifications, similar results can
be stated for $u_{\epsilon}^{\stable}$.  Note also that
Proposition~\ref{prop:usafe-formula} provides a closed-form expression
for the controllers. This allows the closed-loop system to be
implemented without having to continuously solve the
optimization~\eqref{eq:usafe-qp}, which is something one faces
with~\eqref{eq:clf-cbf-qp}, e.g.,~\cite{ADA-SC-ME-GN-KS-PT:19}. The
expression~\eqref{eq:usafe} indicates that smaller $\epsilon$ lead to
controllers with larger norms. Even though here the input is
unconstrained, this should be taken into account in applications with
limited actuation power.

\begin{remark}\longthmtitle{Nominal Controller}
  Our framework can be adapted to the scenario described
  in~\eqref{eq:unom-qp}, where instead of a CLF, one has access to a
  nominal stabilizing controller $u_{\nom}$ and a certificate of
  stability in the form of a Lyapunov function $V$ satisfying
  $L_{f}V(x)+L_{g}V(x)u_{\nom}(x)+W(x)\leq0$ for $x\in\Dc$, with $\Dc$
  some open set.  To design a control $u$ as close as possible to
  $u_{\nom}$ that is safe and stabilizing, one can set
  $v=u-u_{\nom}$. Then, it is easy to check that $V$ is a CLF for
  $\dot{x}=\bar{f}(x)+g(x)v$, where
  $\bar{f}(x)=f(x)+g(x)u_{\nom}(x)$. In this case, one could use the
  safety QP with stability penalty setting
  $a(x)=L_{\bar{f}}V(x)+W(x)$, $b(x)=L_{g}V(x)$,
  $c(x)=-L_{\bar{f}}h(x)-\alpha(h(x))$, and $d(x)=-L_{g}h(x)$.
  \demo
\end{remark}

\section{Analysis of Safety QP with Stability
  Penalty}\label{sec:analysis}

Here, we analyze the closed-loop properties
of~\eqref{eq:control-affine-sys} under $u_{\epsilon}^{\safe}$. We
first show how to choose $\epsilon$ to avoid undesired equilibria of
the closed-loop system and then go on to solve
Problem~\ref{pb:problem-1}. Throughout the
section,
\begin{align*}
  e(x) = -L_{f}h(x)+\frac{1}{\epsilon}L_gh(x)^T L_gV(x)-\alpha(h(x)) .
\end{align*}

\subsection{Ruling out Undesired Equilibrium Points}

Here we show that the closed-loop implementation of the safety QP with
stability penalty controller might introduce new equilibria other than
the origin.  The next result characterizes such equilibria and shows
that, under some conditions, they can be confined to an arbitrarily
small neighborhood of the origin for small enough~$\epsilon$.

\begin{proposition}\longthmtitle{Characterization of
    Equilibria}\label{prop:undes-eq} 
  For $\epsilon>0$, the set of equilibrium points of the closed-loop
  system $\dot{x}=f(x)+g(x)u_{\epsilon}^{\safe}(x)$ in $\mathcal{C}$
  is
  $\mathcal{Q}=\mathcal{Q}_{1}^\epsilon\cup\mathcal{Q}_{2}^\epsilon$,
  with
  \begin{align*}
    \mathcal{Q}_1^{\epsilon}
    &:=\setdef{x\in\mathcal{C}}{e(x)\leq0, \
      f(x)=\frac{1}{\epsilon}g(x)L_gV(x)},
    \\
    \mathcal{Q}_2^{\epsilon}
    &:=\setdef{x\in\partial\mathcal{C}}{{ e(x)}>0, \ f(x) =
      \frac{L_fh(x)}{\norm{L_gh(x)}^2}g(x) L_gh (x)  
    \\
    &
      +
      \frac{g(x)}{\epsilon}(L_gV(x)-\frac{L_gh(x)^T L_gV(x)}{\norm{L_gh(x)}^2}
      L_gh(x))},
  \end{align*}
  and $0\in\mathcal{Q}_{1}^{\epsilon}$. Let
    $\mathcal{V}$ be a neighborhood of the origin, $\bar{\mathcal{V}}$
    a neighborhood of
    $P_{g}:=\setdef{x\in\Cc\backslash\{0\}}{L_gV(x)=0}$ and let
  $N_{1}, N_{2}$, $N_{3}^{\mathcal{V},\bar{\mathcal{V}}}$ and $N_4$ be
  defined by
  \begin{align*}
    N_{1}&:=\sup_{x\in\mathcal{C}} \norm{f(x)}, 
    \\
    N_{2}&:=\sup_{\substack{x\in\partial\Cc \\ e(x)>0}}
    \norm{f(x)-\frac{L_fh(x)}{\norm{L_gh(x)}^2}g(x)L_gh(x)}, 
    \\
    N_{3}^{\mathcal{V},\bar{\mathcal{V}}}&:=\inf_{x\in \mathcal{C\backslash(\mathcal{V}\cup\bar{\mathcal{V}})}} 
                                           \norm{g(x)L_gV(x)}.
    \\
    N_{4}&:=\inf_{\substack{x\in\partial\Cc \\ e(x)>0}}
    \norm{g(x)(L_gV(x) -
    \frac{L_gh(x)^T L_gV(x)}{\norm{L_gh(x)}^2}L_gh(x))}.
  \end{align*}
  then,
  \begin{itemize}
  \item if $N_{1}$ is finite, then
    $\mathcal{Q}_1^\epsilon\subseteq\mathcal{V}$ for all
    $0<\epsilon<\frac{N_3^{\mathcal{V},\bar{\mathcal{V}}}}{N_1}$,
  \item if $N_{2}$ is finite and $N_{4}$ is positive, then
    $\mathcal{Q}_{2}^{\epsilon}=\emptyset$ for
    $0<\epsilon<\frac{N_4}{N_2}$.
  \end{itemize}
\end{proposition}
\begin{proof}
  Since $u_{\epsilon}^{\safe}(x)$ takes a different form depending on
  the sign of $e(x)$, we distinguish two cases:
  \\
  \textbf{Case 1: $e(x)\leq0$}: In this case, the equilibrium points
  of the closed-loop system satisfy
  $f(x)=\frac{1}{\epsilon}g(x)L_{g}V(x)$. Note that if
    $g(x)L_gV(x)=0$, by multiplying on the left by $\nabla V(x)^{T}$
    we obtain $L_gV(x)=0$.  Since $V$ is a CLF, $L_fV(x)<0$ if
    $x\neq0$. This implies that $f(x)\neq0$ and hence $x$ is not an
    equilibrium point. Hence, no point other than the origin satisfies
    $L_{g}V(x)=0$ and $f(x)=\frac{1}{\epsilon}g(x)L_{g}V(x)$, and we
    can choose a neighborhood $\bar{\mathcal{V}}$ of $P_{g}$ with
    $\mathcal{Q}_{1}^{\epsilon}\cap\bar{\mathcal{V}}=\emptyset$. Now,
  by taking any neighborhood $\mathcal{V}$ of the origin, the choice
  $\epsilon<\frac{N_3^{\mathcal{V},\bar{\mathcal{V}}}}{N_1}$ rules out
  any equilibrium of this kind in
  $\mathcal{C}\backslash\mathcal{V}$. Note that, since $f(0)=0$ and
  $\nabla V(0)=0$, we have $e(0)=-\alpha(h(0))\leq0$, and hence
  $0\in\mathcal{Q}_{1}^{\epsilon}$.
  \\
  \textbf{Case 2: $e(x)>0$}: In this case the equilibrium points of
  the closed-loop system satisfy
  \begin{multline}\label{eq:case2-eq}
    f(x) -\frac{L_fh(x)+\alpha(h(x))}{\norm{L_gh(x)}^2}g(x)L_gh(x) =
    \\
    = \frac{g(x)}{\epsilon}(L_gV(x)-\frac{L_gh(x)^T
      L_gV(x)}{\norm{L_gh(x)}^2}L_gh(x)) .
  \end{multline}
  Let us show that these equilibria can only occur in
  $\partial\mathcal{C}$.  Multiplying both sides
  of~\eqref{eq:case2-eq} by $\nabla h(x)^{T}$, we obtain
  $-\alpha(h(x))=0$.  Since $\alpha$ is a class $\mathcal{K}_{\infty}$
  function, this can only occur when $h(x)=0$, i.e.,
  $x\in\partial\mathcal{C}$. Now, by taking
  $\epsilon<\frac{N_4}{N_2}$, all equilibrium points of these kind are
  ruled out.
\end{proof}

Note that the assumption that $N_{1}$ and $N_{2}$ are finite in
Proposition~\ref{prop:undes-eq} is satisfied if $\mathcal{C}$ is
bounded.  The neighborhood $\mathcal{V}$ of the origin
  in the statement can be taken arbitrarily small and,
  consequently, if $N_4$ is positive, the controller
$u_{\epsilon}^{\safe}$ with sufficiently small
$\epsilon$ confines the equilibria of the closed-loop system
arbitrarily close to the origin. However, as $\mathcal{V}$ gets
arbitrarily small, $N_3^{\mathcal{V}}$ (and hence $\epsilon$) could
also get arbitrarily small. In Corollary~\ref{cor:conv-origin} later,
we give sufficient conditions to ensure that this does not happen.

\begin{remark}\longthmtitle{Existence of boundary
    equilibria}\label{rem:bdy-eq}
  The assumption that $N_{4}$ is positive is not satisfied if
  $g(x)L_{g}V(x)$ and $g(x)L_{g}h(x)$ are linearly dependent. In this
  scenario, using condition~\eqref{eq:case2-eq}, we infer that the
  equilibrium points in $\partial\mathcal{C}$ that cannot be removed
  by tuning $\epsilon$ are those where $f(x)$, $g(x)L_{g}V(x)$ and
  $g(x)L_{g}h(x)$ are collinear and $e(x)>0$ for all~$\epsilon$.
  \demo
\end{remark}

\subsection{Incompatibility and Region of Attraction}
Here we show that $u_{\epsilon}^{\safe}$ solves
Problem~\ref{pb:problem-1}. The flexibility provided by the design
parameter $\epsilon$ is instrumental in doing so.  We first introduce
a characterization of points where the CLF and the CBF are
incompatible, the proof of which follows as a special
  case of~\cite[Theorem 1]{XX:18}.

\begin{lemma}\longthmtitle{Characterization of incompatible
    points}\label{lem:incompat-points} 
  Let $\mathcal{D} \subseteq \real^{n}$ be open, $\Cc \subset \Dc$
  closed, $V$ a CLF on $\mathcal{D}$ and $h$ a CBF of~$\Cc$. $V$ and
  $h$ are incompatible at $x\in\Cc$ if and only if $L_{g}V(x)$ and
  $L_{g}h(x)$ are linearly dependent, $L_{g}V(x)^T L_{g}h(x)>0$ and
  $L_{f}V(x)+W(x)>\frac{L_gV(x)^T L_gh(x)}{\norm{L_gh(x)}^2}(L_fh(x)
  +\alpha(h(x)))$.
\end{lemma}

The next result shows that, by taking $\epsilon$ sufficiently small
for the closed-loop system, any level set of $V$ that
does not contain incompatible points is a region of attraction
of a neighborhood of the origin.

\begin{theorem}\longthmtitle{Parameter tuning for guaranteed region of
    attraction}\label{thm:epsilon-roa}
  Let $\mathcal{D} \subseteq \real^{n}$ be open, $\Cc \subset \Dc$
  closed, $V$ a CLF on $\mathcal{D}$ and $h$ a CBF of~$\Cc$. Let
  $\nu>0$ be such that the sublevel set
  $\Gamma_{\nu}=\setdef{x\in \real^n}{V(x)\leq\nu}$ does not contain
  any incompatible points. For $x$ such that $e(x)>0$
    (which implies $L_gh(x)\neq0$ since $h$ is a CBF), define
  {\relsize{-0.95}
    \begin{align*}
      B(x)&:=L_fV(x)\!+\!W(x)\!-\!\frac{L_fh(x)+\alpha(h(x))}{\norm{L_gh(x)}^2}L_gV(x)^T
            L_gh(x),
      \\ 
      C(x)&:=\frac{(L_gV(x)^T
            L_gh(x))^2}{\norm{L_gh(x)}^2}-\norm{L_gV(x)}^2. 
    \end{align*}
  } Let $\mathcal{V}$ be a neighborhood of the origin,
  $\bar{\mathcal{V}}$ a neighborhood of
    $P_{g}:=\setdef{x\in\mathcal{C}\backslash\{0\}}{L_gV(x)=0}$ such
    that $L_{f}V(x)+W(x)\leq0$ for all $x\in\bar{\mathcal{V}}$ and
  $\mathcal{W}$ a neighborhood of
  $P_{\nu}=\setdef{x\in\Gamma_\nu}{e(x)>0, C(x)=0}$
  such that $e(x)>0$ and $B(x)\leq 0$
  for all $x\in\mathcal{W}\backslash\{0\}$. Define constants
  $M_{1}^{\nu}$, $M_{2}^{\nu}$,
  $M_{3}^{\nu,\mathcal{V},\bar{\mathcal{V}}}$ and
  $M_{4}^{\nu,\mathcal{V},\mathcal{W}}$ by
  \begin{align*}
    M_{1}^{\nu}&:=\sup_{x\in\Gamma_\nu}|L_fV(x)+W(x)|,
    \\
    M_{2}^{\nu}
               &:=\sup_{\substack{x\in\Gamma_\nu \\ e(x)>0}}|\frac{L_fh(x)+\alpha(h(x))}{\norm{L_gh(x)}^2}L_gh(x)^T
                 L_gV(x)|,
    \\
    M_{3}^{\nu,\mathcal{V},\bar{\mathcal{V}}}
    &:=\inf_{x\in\Gamma_\nu\backslash(\mathcal{V}\cup\mathcal{\bar{\mathcal{V}}})}\norm{L_gV(x)}^2,
      \\
    M_{4}^{\nu,\mathcal{V},\mathcal{W}}
    &:=\inf_{\substack{x\in\Gamma_\nu\backslash(\mathcal{W}\cup\mathcal{V})\\ e(x)>0}}
      |C(x)|.  
  \end{align*}
  Then, for
  $\epsilon<\bar{\epsilon}:=
  \min\{\frac{M_{4}^{\nu,\mathcal{V},\mathcal{W}}}{M_1^\nu+M_{2}^\nu}
  , \frac{M_{3}^{\nu,\mathcal{V},\bar{\mathcal{V}}}}{M_{1}^\nu}\}$,
  $\mathcal{V}$ is asymptotically stable and $\Gamma_{\nu}\cap \Cc$ is
  forward invariant and a subset of the region of attraction
  of~$\mathcal{V}$.
\end{theorem}
\begin{proof}
  Let $z_{\epsilon}(x):=L_fV(x)+L_gV(x)u_{\epsilon}^{\safe}(x)+W(x)$.
  It follows from~\eqref{eq:usafe} that
  \begin{align*}
    z_{\epsilon}(x)=\begin{cases}
      L_fV(x)+W(x)-\frac{1}{\epsilon}\norm{L_gV(x)}^2  &  \text{if} \
      e(x)\leq0,
      \\
      B(x)+\frac{1}{\epsilon}C(x)  &  \text{if} \
      e(x) >0.
    \end{cases}
  \end{align*}
  We show that $z_{\epsilon}(x)\leq0$ for all
  $x\in\mathcal{C}\backslash\mathcal{V}$ if $\epsilon<\bar{\epsilon}$,
  from which the result follows. First, note that
    $\bar{\mathcal{V}}$ as required in the statement exists because
    $V$ is a CLF and hence, any point $x \neq 0$ that satisfies
    $L_{g}V(x)=0$ is such that $L_{f}V(x)+W(x)<0$ (without loss of
    generality, since if $L_{f}V(x)+W(x)=0$ we can take
    $\tilde{W}(x)=\frac{1}{2}W(x)$). Hence, by continuity there exists
    a neighborhood $\bar{\mathcal{V}}$ of $P_{g}$ where
    $L_fV(x)+W(x)-\frac{1}{\epsilon}\norm{L_gV(x)}^2\leq
    L_fV(x)+W(x)<0$ for all $x\in\bar{\mathcal{V}}$, for any
    $\epsilon>0$. Hence by taking $\epsilon<\bar{\epsilon}$, we ensure that $z_{\epsilon}(x)\leq0$ for all $x\in\bar{\mathcal{V}}$ independently of the sign of $e(x)$.  Note also that $\mathcal{W}$ as required in the
  statement exists because $\Gamma_{\nu}$ does not contain any point
  where $V$ and $h$ are incompatible and therefore by
  Lemma~\ref{lem:incompat-points}, all points in $\Gamma_{\nu}$
  satisfying $C(x)=0$ necessarily also satisfy $B(x)<0$
    (without loss of generality, using a similar argument as
    above). Therefore, by continuity of $B(x)$ for any $\epsilon>0$
  we can take a neighborhood $\mathcal{W}$ around $P_{\nu}$ so that
  $B(x)+\frac{1}{\epsilon}C(x)\leq B(x)\leq0$ for all
  $x\in\mathcal{W}$ (since by Cauchy-Schwartz's inequality,
  $C(x)\leq0$). Hence, by taking $\epsilon<\bar{\epsilon}$, independently of whether $e(x)\leq 0$ or $e(x)>0$ we ensure
  that $z_{\epsilon}(x)\leq0$ for all
  $x\in\mathcal{W}\cup\bar{\mathcal{V}}$.
  Now we argue that if $\epsilon<\bar{\epsilon}$, $z_{\epsilon}(x)\leq0$
  for all
  $x \in \Gamma_{\nu}
  \backslash(\mathcal{W}\cup\mathcal{V}\cup\bar{\mathcal{V}})$. Note
  that
  $\Gamma_{\nu}
  \backslash(\mathcal{W}\cup\mathcal{V}\cup\bar{\mathcal{V}})$ does
  not contain any points where $L_{g}V(x)$ and $L_{g}W(x)$ are
  linearly dependent, since that would imply $C(x)=0$ and hence
  $x\in\mathcal{W}$. Thus, by Cauchy-Schwartz's inequality, $C(x)<0$
  for all
  $x\in\Gamma_{\nu}
  \backslash(\mathcal{W}\cup\mathcal{V}\cup\bar{\mathcal{V}})$. Hence,
  $M_{4}^{\nu,\mathcal{V},\mathcal{W}}>0$.  Note also that
  $M_3^{\nu,\mathcal{V},\bar{\mathcal{V}}}>0$. Therefore, regardless
  of whether $e(x)\leq0$ or $e(x)>0$, by taking
  $\epsilon<\bar{\epsilon}$ we ensure that $z_{\epsilon}(x)\leq0$ for
  all
  $x\in\Gamma_{\nu}\backslash(\mathcal{W}\cup\mathcal{V}\cup\bar{\mathcal{V}})$,
  as claimed.  Moreover, since by construction $u_{\epsilon}^{\safe}$
  satisfies~\eqref{eq:cbf-ineq} and is Lipschitz, by~\cite[Theorem
  2]{ADA-SC-ME-GN-KS-PT:19}, trajectories stay inside $\mathcal{C}$
  for all $t\geq0$.
\end{proof}

Note that in the statement of
  Theorem~\ref{thm:epsilon-roa}, one can pick $\mathcal{V}$
arbitrarily small, which might require an arbitrarily small
$\epsilon$. The next result states that under some
  additional reasonable assumptions,
this does not happen and hence there exists a finite $\epsilon$ for
which trajectories converge to the origin.

\begin{corollary}\longthmtitle{Convergence to the
    origin}\label{cor:conv-origin}
  Under the same assumptions and notation of
  Theorem~\ref{thm:epsilon-roa}, assume additionally that
    $f,g\in\Cc^{1}(\real^n)$, $V\in\Cc^{2}(\real^n)$,
    $0\in\textrm{int}(\Cc)$ and
    $\ker(g(0)^T)\subseteq \mathcal{V}_{s}(\frac{\partial f}{\partial
      x}(0))$.  Then, for
    $\epsilon<\hat{\epsilon}:=
    \min\{\frac{\bar{\lambda}_{\min}(g(0)g(0)^T\nabla^2
      V(0))}{|\lambda_{\max}(\frac{\partial f}{\partial x}(0))|},
    \bar{\epsilon}\}$, the origin is asymptotically stable and
    $\Gamma_{\nu}\cap C$ is forward invariant and a subset of the
    region of attraction of the origin.
\end{corollary}
\begin{proof}
    Since $0\in\textrm{int}(\Cc)$, $e(0)<0$ and the
    Jacobian of the closed-loop system evaluated at $0$ is
    $J=\frac{\partial f}{\partial
      x}(0)-\frac{1}{\epsilon}g(0)g(0)^{T}\nabla^{2}V(0)$.  We show
    that, with $\epsilon<\hat{\epsilon}$, one has $v^TJv<0$ for
    $v\in\real^n\backslash\{0\}$.  First, consider
    $v\in\ker(g(0)^T)$. By assumption,
    $v\in\mathcal{V}_s(\frac{\partial f}{\partial x}(0))$, and hence
    $v^TJv=v^{T}\frac{\partial f}{\partial x}(0)v<0$. Now, assume
    $v\notin\ker(g(0)^T)$. Since $\nabla^2 V(0)$ is positive definite
    and $g(0)g(0)^T$ is positive semidefinite,
    $\ker(g(0)g(0)^{T}\nabla^{2}V(0))=\ker(g(0)g(0)^T)$ and
    $g(0)g(0)^{T}\nabla^{2}V(0)$ has non-negative
    eigenvalues~\cite[7.2.P21]{RAH-CRJ:12}. Hence,
    $v^TJv \leq (\lambda_{\max}(\frac{\partial f}{\partial
      x}(0))-\frac{1}{\epsilon}\bar{\lambda}_{\min}(g(0)g(0)^T\nabla^2
    V(0))) \|v\|^{2}$. This implies that $J+J^T$ is
      negative definite, and since the real parts of its eigenvalues
      are twice those of $J$, we obtain that $J$ is Hurwitz.
    Therefore, we can take $\mathcal{V}$ in
    Theorem~\ref{thm:epsilon-roa} such that the closed-loop
    trajectories with
    $\epsilon <
    \frac{\bar{\lambda}_{\min}(g(0)g(0)^T\nabla^2V(0))}{|\lambda_{\max}(\frac{\partial
        f}{\partial x}(0))|}$ starting at $\mathcal{V}$ converge
    to~$0$.  Finally, reasoning as in Theorem~\ref{thm:epsilon-roa},
  $V$ is decreasing on $\Gamma_{\nu}\backslash\mathcal{V}$, and the
  result follows.
\end{proof}

Under the assumptions of Corollary 5.4, by ensuring that the origin is
asymptotically stable in $\Gamma_{\nu}$, we rule out the existence of
equilibrium points in $\Gamma_{\nu}$ other than the origin. If the
conditions of Corollary~\ref{cor:conv-origin} are not satisfied or
$\epsilon\geq\hat{\epsilon}$, other undesired behaviors like limit
cycles or convergence to undesired equilibria like the ones found in
Proposition~\ref{prop:undes-eq} cannot be ruled out.
Theorem~\ref{thm:epsilon-roa} and Corollary~\ref{cor:conv-origin}
solve Problem~\ref{pb:problem-1}. Under the stated assumptions, by
taking $u_{\epsilon}^{\safe}$ with $\epsilon<\hat{\epsilon}$ as a safe
stabilizing controller, an inner approximation of the region of
attraction of the origin is the largest level set of $V$ that does not
contain any incompatible points inside it.  In particular, if there
exists a sublevel set of $V$ that contains $\Cc$,
$u_{\epsilon}^{\safe}$ with $\epsilon<\hat{\epsilon}$ safely
stabilizes the origin and the whole safe set $\mathcal{C}$ is in its
region of attraction.

\section{Simulations}\label{sec:sims}
Here, we compare the stability QP with safety penalty controller with
the CLF-CBF QP~\eqref{eq:clf-cbf-qp} and its modification, M-CLF-CBF
QP, introduced in~\cite[Theorem~3]{XT-DVD:21} to avoid undesired
equilibria. We focus on the following planar system
\begin{align}\label{eq:example}
  \begin{pmatrix} \dot{x}_1 \\ \dot{x}_2 \end{pmatrix}
  =\begin{pmatrix} x_1 \\ x_2 \end{pmatrix} +
  \begin{pmatrix} 1 & 0 \\ 0 & 1 \end{pmatrix}u.
\end{align}
For this system,
$V(x_1,x_2) = \frac{1}{2}x_{1}^{2}+\frac{1}{2}x_{2}^{2}$ is a CLF.
The safe set $\Cc$ is the complement of the ball
$\setdef{x\in\real^2}{\norm{x-(0,4)}\leq 2}$, and we use the CBF
$h(x_1,x_2)=x_{1}^{2}+(x_2-4)^{2}-4$, with $\alpha(s)=s$.  According
to~\cite{XT-DVD:21}, the CLF-CBF QP~\eqref{eq:clf-cbf-qp} creates
undesired equilibria in $\text{int}(\mathcal{C})$ for all values of
$p$. Instead, both M-CLF-CBF QP and the stability QP with safety
penalty controller $u_{\epsilon}^{\safe}$, with $\epsilon\neq1$, do
not introduce undesired equilibria in $\text{int}(\mathcal{C})$.  The
latter can be checked from the definition of
$\mathcal{Q}_{1}^{\epsilon}$ given in Proposition~\ref{prop:undes-eq}.
In this example, the incompatible points are given by
$\setdef{(x_1,x_2)\in\real^2}{x_1=0,x_2>4}$. Therefore, the
approximation of the region of attraction given by
Theorem~\ref{thm:epsilon-roa} is
$\Gamma_{2} = \setdef{x\in\real^2}{\norm{x}^2<4}$.
Figure~\ref{fig:example} shows that the stability QP with safety
penalty controller and M-CLF-CBF QP behave similarly, whereas CLF-CBF
QP~\eqref{eq:clf-cbf-qp} fails to stabilize the origin. The plot also
illustrates that trajectories starting at $(0,9)$ converge to the
boundary equilibrium point at $(0,6)$ for all three approaches (this
corresponds to a point where $f, gL_{g}V$, and $gL_{g}h$ are
collinear, cf. Remark~\ref{rem:bdy-eq}). This is not
  surprising since, for scenarios where the unsafe set is bounded,
  global convergence with a smooth vector field is impossible due to
  topological obstructions~\cite{PB-CMK:17}. An advantage of the
approach proposed here is the explicit inner approximation of the
region of attraction which, as Figure~\ref{fig:example} shows, is
conservative.

\begin{figure}[htb]
  \centering
  \includegraphics[width=0.4\textwidth]{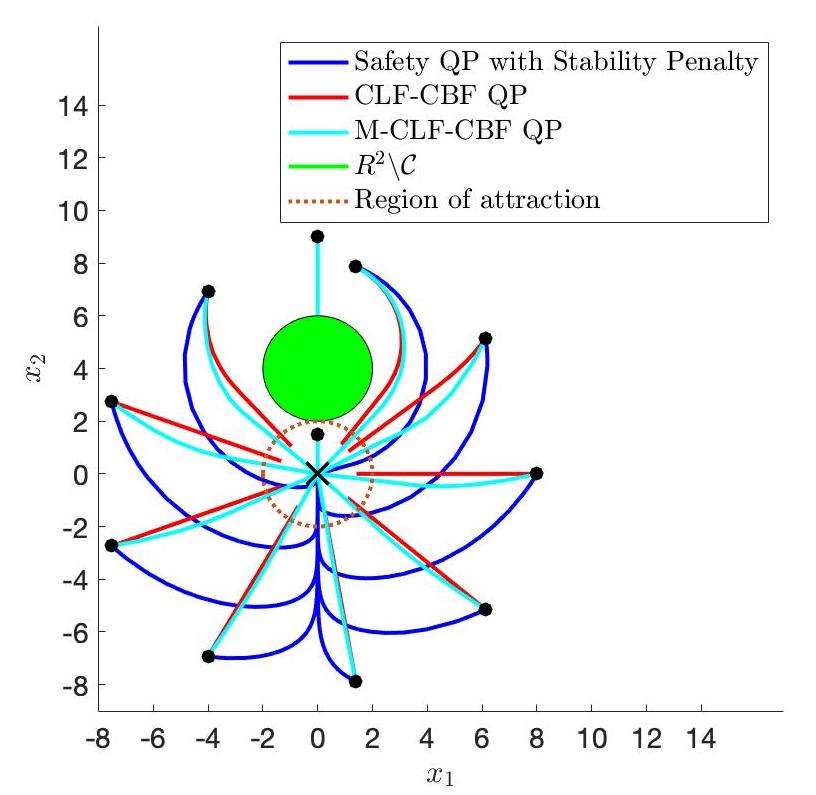}
  \caption{Safe stabilization of a planar system.  The green ball is
    the set of unsafe states and the small dots display ten initial
    conditions for the system trajectories under the CLF-CBF QP, the
    M-CLF-CBF QP, and the safety QP with stability penalty
    controllers. The orange dotted curve marks the boundary of the
    estimate $\Gamma_{2}$ of the region of attraction. The CLF-CBF QP
    controller (with $p=1$) preserves safety but does not reach the
    origin because of undesired equilibrium points.  The safety QP
    with stability penalty (with $\epsilon=0.01$) and the M-CLF-CBF QP
    (with $p=1$) preserve safety and have trajectories converge to the
    origin, except for the one starting at $(0,9)$.}\label{fig:example}
\end{figure}

\section{Conclusions}\label{sec:conclusions}
We have addressed the problem of safe stabilization of nonlinear
affine control systems by proposing an optimization-based feedback
design framework inspired by penalty methods for constrained
optimization. Our design enforces strictly either stability or safety
via a hard constraint while promoting the satisfaction of the other
property via a soft constraint. We have characterized the equilibria
of the closed-loop system under the proposed controllers. We have
shown how to tune the penalty parameter to eliminate spurious
equilibria and to increase the region of attraction to all Lyapunov
level sets that do not include points where the CLF and the CBF are
not compatible. Future work will develop tighter estimates of the
region of attraction, consider extra design parameters and explore the
extension of the proposed framework to generalized notions of CBFs.

\end{document}